\documentclass[a4paper]{article}

\usepackage[english]{babel}
\usepackage[T1]{fontenc}
\usepackage[utf8]{inputenc}
\usepackage{lmodern}

\usepackage[a4paper,margin=3cm]{geometry}

\usepackage{amsmath}
\usepackage{graphicx}
\usepackage{amsthm}
\usepackage{amsmath}
\usepackage{amssymb}
\usepackage{amsfonts}
\usepackage{multirow}
\usepackage{enumerate}
\usepackage{hyperref}

\usepackage{lscape}
\usepackage{longtable}
\usepackage{colortbl}
\usepackage{multirow}
\usepackage[table]{xcolor}

\newtheorem{conj}{Conjecture}

\newtheorem{thm}{Theorem}
\newtheorem*{thm*}{Theorem}

\newtheorem{lem}[thm]{Lemma}

\theoremstyle{definition}

\newtheorem{clm}[thm]{Claim}
\newtheorem*{clm*}{Claim}

\newtheorem*{obs*}{Observation}

\newcommand{\flagmatic}{\texttt{flagmatic}}

\newcommand{\abs}[1]{\left\vert#1\right\vert}

\DeclareMathOperator{\dd}{d\!}
\allowdisplaybreaks

\def\D#1#2{\raisebox{-2.5pt}{\includegraphics{D-#1-#2.mps}}}
\def\graf#1{\D{4}{#1}}
\def\result#1#2#3#4#5#6{%
\stepcounter{grafik}\arabic{grafik} & #1 & #3 & #4 & #5 & #6 \\\hline}
\def\ss{$\ \ $}

\def\Ilim#1{I(#1)}

\def\Nnum#1#2{N(#1,#2)}
\def\Gconstr#1{G^{#1}}

\def\Tn#1{T_{#1}}
\def\Sn#1{\overrightarrow{S}_{#1}}

\def\Pn#1{\overrightarrow{P}_{#1}}
\def\Kn#1{\overrightarrow{K}_{#1}}
\def\An#1{I_{#1}}
\def\S1#1{S^1(#1)}
\def\Grand#1{#1^{\mathrm{rand}}}
\def\Treg#1{T_{#1}^{\mathrm{reg}}}

\newcounter{grafik}
\long\def\uboundflagmatic#1#2{Upper bound (by \flagmatic{} on #1 vertices): \colorbox{gray!10}{#2}}

\long\def\grafsekcja#1#2#3#4#5#6#7{%

\subsection{Graph #1}\label{subsec:graph#1}

Graph: #2

\medskip
\noindent #3

\medskip
\noindent Lower bound: \colorbox{gray!10}{#4}#5

\medskip
\noindent Construction:

\smallskip
\centerline{\includegraphics[scale=#6]{construction-#1.mps}}
\smallskip

\noindent #7}

\usepackage{listings}
\usepackage{xcolor}

\definecolor{codegreen}{rgb}{0,0.6,0}
\definecolor{codegray}{rgb}{0.5,0.5,0.5}
\definecolor{codepurple}{rgb}{0.58,0,0.82}
\definecolor{backcolour}{rgb}{0.95,0.95,0.92}

\lstdefinestyle{mystyle}{
    backgroundcolor=\color{backcolour},   
    commentstyle=\color{codegreen},
    keywordstyle=\color{magenta},
    numberstyle=\tiny\color{codegray},
    stringstyle=\color{codepurple},
    basicstyle=\ttfamily\footnotesize,
    breakatwhitespace=false,         
    breaklines=true,                 
    captionpos=b,                    
    keepspaces=true,                 
    numbers=left,                    
    numbersep=5pt,                  
    showspaces=false,                
    showstringspaces=false,
    showtabs=false,                  
    tabsize=2
}

\begin{document}

\title{On the inducibility of oriented graphs on four vertices\thanks{This work was supported by the National Science Centre grant 2016/21/D/ST1/00998.}}

\author{\L{}ukasz Bo\.zyk\thanks{Faculty of Mathematics, Informatics, and Mechanics, University of Warsaw, Banacha 2, 02-097 Warszawa, Poland. E-mail: {\tt l.bozyk@uw.edu.pl}.} 
\and Andrzej Grzesik\thanks{Faculty of Mathematics and Computer Science, Jagiellonian University, {\L}ojasiewicza 6, 30-348 Krak\'{o}w, Poland. E-mail: {\tt Andrzej.Grzesik@uj.edu.pl}.}
\and Bart\l{}omiej Kielak\thanks{Faculty of Mathematics and Computer Science, Jagiellonian University, {\L}ojasiewicza 6, 30-348 Krak\'{o}w, Poland. E-mail: {\tt bartlomiej.kielak@doctoral.uj.edu.pl}.}}

\date{}

\maketitle

\begin{abstract}
We consider the problem of determining the inducibility (maximum possible asymptotic density of induced copies) of oriented graphs on four vertices. 
We provide exact values for more than half of the graphs, and very close lower and upper bounds for all the remaining ones. It occurs that, for some graphs, the structure of extremal constructions maximizing density of its induced copies is very sophisticated and complex.  
\end{abstract}

\section{Introduction}

Fix a~graph $H$ on $k$ vertices. For a~graph $G$, let $\Nnum{H}{G}$ denote the number of induced copies of~$H$ in $G$, i.e., the number of $k$-sets of vertices of $G$ that induce graphs isomorphic to $H$. 
Define 
$$\Ilim{H} = \lim_{n \to \infty} \frac{\max\{\Nnum{H}{G} \,:\, \abs{V(G)} = n\}}{\binom{n}{k}}.$$
We call $\Ilim{H}$ the \emph{inducibility} of a~graph $H$. For $H$ being complete graph or independent set we have $\Ilim{H} = 1$, but in most of the remaining cases the problem of determining $\Ilim{H}$ is open.
The concept of inducibility was introduced in 1975 by Pippenger and Golumbic \cite{PG}, who observed that
$$\Ilim{H} \geq \frac{k!}{k^k - k}$$
for any graph $H$ on $k$ vertices, and conjectured the following:
\begin{conj}\label{conj:Pippenger-Golumbic}
If $H$ is a~cycle of length $k \geq 5$, then $\Ilim{H} = \frac{k!}{k^k-k}$.
\end{conj}

Since then, many results on inducibility were obtained, including the inducibility of complete bipartite graphs~\cite{BNT, BS, PG}, multipartite graphs~\cite{BEHJ}, and blow-ups of graphs~\cite{HHN}. Recently, also weaker bounds in Conjecture~\ref{conj:Pippenger-Golumbic} were obtained ~\cite{HT, KNV} and the conjecture was proved in the case $k = 5$ by Balogh, Hu, Lidick\'y, and Pfender~\cite{BHLP}.
Inducibility was determined exactly for all graphs on four vertices~\cite{EL, Hi} with a~single exception of the four-vertex path $P_4$, for which we still do not even have a~conjecture about what should be the extremal construction. The best lower bound is $\approx 0.2014$, while the best upper bound $\approx 0.2045$ was obtained by Vaughan using his \flagmatic{} software~\cite{Flagmatic}.

Inducibility can be defined in the same way also for other discrete structures, including hypergraphs, multigraphs, directed graphs, and oriented graphs. In our work we focus on the latter. Recall that by an oriented graph we mean a~graph with directed edges which has no loops, no multiple edges, and no bidirected edges (2-cycles), i.e., a simple graph with an orientation of the edges. 

In general, little is known about the inducibility of oriented graphs. Huang \cite{Hu} determined the inducibility of oriented stars $\Sn{k}$ for all $k \geq 2$ and of complete bipartite digraphs $\Kn{s,t}$ for $2 \leq s \leq t$. Recently, Hu et al. \cite{HMNW} determined the inducibility of all other orientations of stars on at least~$7$ vertices. There is also an analogue of Conjecture~\ref{conj:Pippenger-Golumbic} for oriented graphs:

\begin{conj}\label{conj:directed_cycles}
If $H$ is a~directed cycle of length $k \geq 4$, then $\Ilim{H} = \frac{k!}{k^k-k}$.
\end{conj}

The conjectured value is achieved in the sequence of iterated copies of a directed $k$-cycle. 
The case $k=4$ was proved by Hu, Lidick\'y, Pfender, and Volec~\cite{HLPV}, who also determined inducibility of all other orientations of $C_4$. Note that the case $k=5$ follows from the result on undirected cycles~\cite{BHLP}. 
In \cite{CLP} Choi, Lidick\'y, and Pfender made a similar conjecture, that inducibility of a directed path on $k$ vertices is also achieved in the sequence of iterated copies of a directed $k$-cycle. 

The inducibility of all oriented graphs on $3$ vertices is known. For directed triangle it is $1/4$ and is achieved by a sequence of any regular tournaments. The inducibility of a graph with one edge is $3/4$, since it is the same as the inducibility of unoriented complement of $K_{1,2}$ solved in~\cite{PG}, and is attained by a sequence of unions of two isolated arbitrary tournaments of equal size. The inducibility of a directed star was determined by Huang~\cite{Hu} to be $2\sqrt{3}-3$ and is achieved in an iterative construction. The case of a path on $3$ vertices was announced in~\cite{CLP} to be solved by Hladk\'y, Kr\'al' and Norin---the inducibility is equal to $2/5$ and is achieved in the sequence of iterated copies of a directed $4$-cycle.

In this work, we consider the inducibility of all oriented graphs on four vertices. Up to isomorphism, there are 42 such graphs. Since $\Ilim{H} = \Ilim{\overleftarrow{H}}$, where $\overleftarrow{H}$ is the~graph obtained from $H$ by reversing all arcs, the number of non-isomorphic cases to consider can be reduced to 30. 

In Section \ref{sec:graphs} we present upper bounds and constructions providing lower bounds for all oriented graph on four vertices. 
Few of them follow from the results on the inducibility of undirected graphs \cite{EL}, whereas star~\cite{Hu}, tournaments~\cite{BLPP} and orientations of $C_4$~\cite{HLPV} were proven earlier. 
In the remaining cases, the upper bounds were obtained mostly using \flagmatic{} software \cite{Flagmatic} based on the flag algebra method of Razborov~\cite{Raz07}. 
In several cases, presented constructions give lower bounds that are matching the upper bounds, while in the remaining ones, when the construction is complex, the lower bounds differ by at most $0.004$ in one case, and by at most $0.001$ in all the other cases. This indicates that the constructions might be correct and the applied flag algebras computations were not enough to obtain the matching upper bounds. 
All of the results are summarized in Table~\ref{tab:summary}. Whenever the constant is irrational, it is defined in the appropriate subsection of Section \ref{sec:graphs}. Description of the used notation and explanations of the pictograms applied to illustrate the constructions are contained in Section~\ref{sec:preliminaries}.

It is worth mentioning that the obtained results indicate the structure of constructions giving the inducibility to be far richer than the intuition suggests. Yuster~\cite{Yu} and independently Fox, Huang, and Lee~\cite{FHL} proved that for almost all graphs $H$, the inducibility is attained by the iterated blow-up of~$H$. For small graphs the situation is different. Out of $28$ non-trivial non-isomorphic graphs considered here, only $2$ of them have this property, whereas such constructions as those in Subsections \ref{subsec:graph4}, \ref{subsec:graph20}, or \ref{subsec:graph23} show that the inducibility can be attained by very sophisticated and complex structures. 

\setcounter{grafik}{0}
\def\resultex#1#2#3#4#5#6{%
\stepcounter{grafik}\arabic{grafik} & #1 & #3 & $=$ & #5 & #4 & #6 \\\hline}
\def\ss{$\ \ $}
\def\result#1#2#3#4#5#6{%
\stepcounter{grafik}\arabic{grafik} & #1 & #3 & \cellcolor{black!10}$>$ & #5 & #4 & #6 \\\hline}
\def\ss{$\ \ $}

\begin{table}
\centering
\renewcommand{\arraystretch}{1.65}
\begin{tabular}{|c|c|c|c|ccc|}
\hline
\rowcolor{black!10}
\multirow{2}{*}{Id} & \multirow{2}{*}{$H$, $\overleftarrow{H}$} & \multirow{2}{*}{Upper bound} &  & \multicolumn{3}{c|}{Lower bound} \\\cline{5-7}
\rowcolor{black!10}
\multirow{-2}{*}{Id} &\multirow{-2}{*}{$H$, $\overleftarrow{H}$} & \multirow{-2}{*}{Upper bound}  && Value & Approximation & Construction \\\hline
\resultex{\graf{1}}{--}{1}{1}{$1$}{$\An{n}$}
\resultex{\graf{2}}{7}{$72/125$}{0.576}{$72/125$}{$\An{5}\odot T_n$}
\resultex{\graf{3}}{4}{$3/8$}{0.375}{$3/8$}{$T_n\sqcup T_n$}
\result{\graf{4}\ss\graf{5}}{6}{$\approx$ 0.235046}{0.234309}{$c_4$}{$\Gconstr{4}$}
\result{\graf{6}}{6}{$\approx$ 0.204123}{0.203175}{64/315}{$\overrightarrow{C_4}^{\odot n} \sqcup \overrightarrow{C_4}^{\odot n}$}
\result{\graf{7}}{6}{$\approx$ 0.193552}{0.193548}{$6/31$}{$\overrightarrow{C_5}^{\odot n}$}
\result{\graf{8}\ss\graf{9}}{6}{$\approx$ 0.105867}{0.102124}{$c_7$}{$\Gconstr{7}$}
\resultex{\graf{10}}{6}{81/512}{0.158203}{$81/512$}{$\An{n}\stackrel{3/4}{\to} \An{n}$}
\resultex{\graf{11}\ss\graf{12}}{5}{$c_9$}{0.423570}{$c_9$}{$\Gconstr{9}$}
\resultex{\graf{13}\ss\graf{14}}{5}{$81/400$}{0.2025}{$81/400$}{$\Gconstr{10}$}
\resultex{\graf{141}}{5}{$1/8$}{0.125}{$1/8$}{$T_n^{\mathrm{reg}}\sqcup T_n^{\mathrm{reg}}$}
\resultex{\graf{142}}{4}{$1/2$}{0.5}{$1/2$}{$T_n\sqcup T_n$}
\resultex{\graf{15}}{5}{2/21}{0.095238}{$2/21$}{$\overrightarrow{C_4}^{\odot n}$}
\resultex{\graf{16}}{5}{$3/16$}{0.1875}{$3/16$}{$K_{n,n}^{\mathrm{rand}}$}
\resultex{\graf{17}}{--}{$c_{15}$}{0.189000}{$c_{15}$}{$\Gconstr{15}$}
\resultex{\graf{18}}{5}{$3/8$}{0.375}{$3/8$}{$\overrightarrow{K_{n,n}}$}
\result{\graf{19}\ss\graf{20}}{6}{$\approx$ 0.095640}{0.095238}{2/21}{$H^{\odot n}$}
\result{\graf{21}\ss\graf{22}}{6}{$\approx$ 0.189030}{0.189000}{$c_{18}$}{$\Gconstr{18}$}
\result{\graf{23}\ss\graf{24}}{6}{$\approx$ 0.317681}{0.317678}{$c_{19}$}{$\Gconstr{19}$}
\result{\graf{25}\ss\graf{26}}{6}{$\approx$ 0.119760}{0.119537}{$c_{20}$}{$\Gconstr{20}$}
\resultex{\graf{27}}{5}{$c_{21}$}{0.227173}{$c_{21}$}{$S^1(\frac{9+\sqrt{3}}{26})$}
\result{\graf{28}}{9}{$\approx$ 0.244055}{0.244053}{$c_{22}$}{$\Gconstr{22}$}
\result{\graf{29}\ss\graf{30}}{6}{$\approx$ 0.177784}{0.177630}{$c_{23}$}{$\Gconstr{23}$}
\resultex{\graf{31}\ss\graf{32}}{5}{$3/8$}{0.375}{$3/8$}{$T_n\rightarrow \An{n}$}
\resultex{\graf{33}}{5}{$4/9$}{0.444444}{$4/9$}{$\overrightarrow{C_3}\odot \An{n}$}
\result{\graf{34}\ss\graf{35}}{6}{$\approx$ 0.113205}{0.112567}{$c_{26}$}{$\Gconstr{26}$}
\result{\graf{36}}{6}{$\approx$ 0.148148}{0.148148}{4/27}{$S^1(4/9)$}
	\resultex{\graf{38}\ss\graf{39}}{8}{$c_{28}$}{0.157501}{$c_{28}$}{$\Gconstr{28}$}
\resultex{\graf{37}}{4}{$1/2$}{0.5}{$1/2$}{$S^1(1/2)$}
\resultex{\graf{40}}{--}{1}{1}{$1$}{$T_n$}
\end{tabular}
\caption{Summary of bounds on inducibility of graphs on four vertices.}\label{tab:summary}
\end{table}

\section{Preliminaries}\label{sec:preliminaries}

For a vertex $v\in G$, denote by $d^+(v)$, $d^-(v)$, and $d'(v):=|G|-1-d^+(v)-d^-(v)$ its outdegree, indegree, and nondegree, respectively.

By $\An{n}$, $\Tn{n}$, and $\Treg{n}$ we denote, respectively, the empty graph, the transitive tournament, and any regular tournament on $n$ vertices.
For a graph~$G$, the oriented graph $\Grand{G}$ denotes a~graph obtained from $G$ by orienting every edge independently at random.

By $G \odot H$ we denote the lexicographic product of $G$ and $H$. We also write $G^{\odot n}$ to indicate a~graph $G \odot \ldots \odot G$, an iterated blow-up of $G$.

For any graphs $G$ and $H$, $G \sqcup H$ denotes the graph obtained by taking disjoint union of graphs $G$ and $H$. If $G$ and $H$ are oriented graphs, then we write $G \stackrel{p}{\to} H$ to denote a~graph constructed from $G \sqcup H$ by joining each vertex of~$G$ with each vertex of~$H$ by an arc independently at random with probability~$p$. For $p=1$, we just write $G \to H$.

For any $\alpha \in [0, \frac 12]$, we define a~\emph{circular} graph $\S1{\alpha}$ with vertex set $[0,1)$ and edge from $x$ to $x + a \mod 1$ for each $x \in [0,1)$ and each $a \in [0,\alpha)$.

While depicting constructions corresponding to obtained lower bounds, we use the following conventions. First of all, the illustrations show the limit structure of each construction instead of the finite graphs forming a sequence giving the lower bound. This allows to see the structure of this sequence without caring about getting integer sizes of particular blobs.
To each blob is assigned a real number from the interval $(0,1)$ which corresponds to the fraction of vertices present in that blob in the limit (in a~corresponding $n$-vertex graph, a~blob of size $\alpha$ is assumed to have roughly $\alpha n$ vertices). 
If no blob sizes occur, all parts in the picture are meant to have equal sizes.
Every blob forming an independent set is depicted as a white (empty) circle, while a blob forming an arbitrary tournament is depicted with a gray (shaded) circle. When some other structure appears inside the blob, then there is a letter indicating the structure. In particular, {\bf T} for the transitive tournament and {\bf R} for the random tournament. Iterated constructions are marked by a dot inside the circle---this means that the blob consists of a copy of the entire construction. We also provide a special pictogram for a circular graph~$S^1(\alpha)$. All the above pictograms are summarized in Figure~\ref{fig:graphs}.

\begin{figure}[ht]
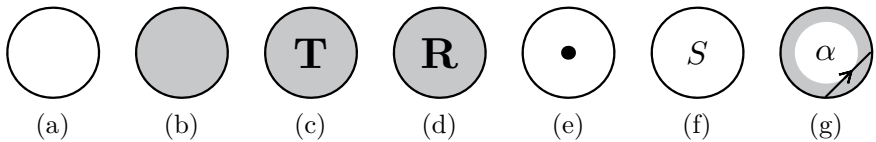

\centering\begin{tabular}{ccccccc}
\includegraphics[scale=.82]{construction-101.mps} & 
\includegraphics[scale=.82]{construction-102.mps} & 
\includegraphics[scale=.82]{construction-104.mps} & 
\includegraphics[scale=.82]{construction-103.mps} & 
\includegraphics[scale=.82]{construction-105.mps} & 
\includegraphics[scale=.82]{construction-107.mps} & 
\includegraphics[scale=.82]{construction-106.mps} \\
(a) & (b) & (c) & (d) & (e) & (f) & (g) \\
\end{tabular}
\caption{(a) Anticlique. (b) Arbitrary tournament. (c) Transitive tournament. (d) Random tournament. (e) Iterative structure. (f) Blob with structure $S$. (g)~Circular graph with parameter~$\alpha$.}\label{fig:graphs}
\end{figure}

\section{Graphs}\label{sec:graphs}

In the following subsections, for each oriented graph on $4$ vertices we present the proven upper and lower bound on its inducibility and provide schematic picture and description of a~construction giving the lower bound. If a~pair $H, \overrightarrow{H}$ is considered, only a~construction for the~graph $H$ is presented; by reversing arcs, one may obtain an analogous construction for the~graph $\overrightarrow{H}$.

Each value of the upper bound which is not preceded with an approximation symbol ($\approx$) is proven exactly and meets the lower bound. 
Whenever we use \flagmatic{} software~\cite{Flagmatic} we indicate on how many vertices it is performed. Unless it is possible to make some reduction by forbidding certain structures in the extremal construction, the computations are performed using graphs on at most $6$ vertices. It is possible to increase this number, which will result in slight improvement of the upper bounds, but will cause much longer running time of the program. 

In almost all cases with sharp bounds, the standard application of the flag algebras semidefinite method is insufficient, due to rounding errors made in the rationalization process. In order to overcome this, one need to provide additional eigenvectors for eigenvalue $0$ of the numerically obtained semidefinite matrix, which are not implied by the extremal construction. Such eigenvectors were found using the method described in the appendix of \cite{BHLPUV} and in Section 2.4.2.2 of \cite{Baber}. For each calculation we publish the applied \flagmatic{} code with all commands and added eigenvectors, as well as a certificate useful for verification of the obtained bound. Explanation how to understand and use the codes is written in Appendix~\ref{sec:code}, while details on verification of the computer-assisted proofs are written in Appendix~\ref{sec:inspect_certificate}.
All the codes and certificates are available with the electronic preprint of this manuscript on \hyperlink{https://arxiv.org/abs/2010.11664}{arXiv:2010.11664}.

For two graph we use the following version of Lemma 2.1 from \cite{Hu}, which can be proved using essentially the same method:

\begin{lem}\label{lem:symmetrization_lemma}
Let $H$ be an oriented graph with the following property --- for every $v, w \in V(H)$ not joined by an arc, $N^+(v) = N^+(w)$ and $N^-(v) = N^-(w)$. Then, for any $n \in \mathbb{N}$, there exists an oriented graph $G$ on $n$ vertices satisfying the same property and such that $N(H, G) = \max \{N(H, G'): \abs{V(G')} = n\}$.
\end{lem}

\grafsekcja{1}{\graf{1}}{%
Upper bound: \colorbox{gray!10}{1}}{%
1}{}{1}{Anticlique.}

\grafsekcja{2}{\graf{2}}{%
Upper bound: \colorbox{gray!10}{72/125}

As each weak component of this graph is a transitive tournament, the inducibility of this graph is equal to the inducibility of the non-oriented graph $K_2\sqcup \An{2}$ (an edge on four vertices), which was explicitly determined in \cite{EL}.}{%
72/125}{}{.6}{Balanced union of arbitrary five tournaments.
}

\grafsekcja{3}{\graf{3}}{
Upper bound: \colorbox{gray!10}{3/8}

As each weak component of this graph is a transitive tournament, the inducibility of this graph is equal to the inducibility of the non-oriented graph $K_2\sqcup K_2$, which is equal to the inducibility of its complement $K_{2,2}$, whose value is known \cite{BNT, PG}.}{%
3/8}{}{1}{Balanced union of arbitrary two tournaments.}

\grafsekcja{4}{\graf{4}\ss\graf{5}}{%
\uboundflagmatic{6}{$\approx$ 0.235046}}{%
$\approx$ 0.234309}{
obtained by blob size optimization, assuming that $x_i = y_i$ for all $i \geq 1$, $x_i' = y_i'$ for all $i \geq 0$, $x_1' = x_i' = 0$ for $i \geq 6$, and $(x_i)_{i \geq 5}$ is a~geometric series.
}{.5}{Construction $\Gconstr{4}$ is the following: split vertices into parts $X_i$ and $Y_i$, $i \geq 0$. Consider also partitions $X_0 = \bigcup_{i=0}^\infty X_i'$ and $Y_0 = \bigcup_{i=0}^\infty Y_i'$. Add all edges from $X_j$ to $X_i$ and from $Y_j$ to $Y_i$ for $0 \leq i < j$. Finally, add all edges from $X_i$ to $\bigcup_{j=1}^i Y_i'$ and from $Y_i$ to $\bigcup_{j=1}^i X_i'$ for $i \geq 1$.}

\grafsekcja{5}{\graf{6}}{%
\uboundflagmatic{6}{$\approx$ 0.204123}}{%
$64/315 \approx 0.203175$}{}{1}{Balanced union of two iterated balanced blow-ups of $\overrightarrow{C_4}$.}

\grafsekcja{6}{\graf{7}}{%
\uboundflagmatic{6}{$\approx$ 0.193552}}{%
$6/31 \approx 0.193548$ 
}{
}{.7}{Iterated balanced blow-up of $\overrightarrow{C_5}$.}

\grafsekcja{7}{\graf{8}\ss\graf{9}}{%
\uboundflagmatic{6}{$\approx$ 0.105867}
}{%
$c_7 \approx 0.102124$}{, where 
$$c_7 = \max_{\substack{a,b,c,d\in[0,1] \\ 2a+2b+2c+d=1}} \frac{24((a+b)^2c^2+2abd(b+2c))}{1-2a^4-2b^4-2c^4-d^4}$$
and the maximum is attained at $(a, b, c, d) \approx (0.117446, 0.159343, 0.146896, 0.152629)$.
}{.9}{Construction $\Gconstr{7}$ is a~weighted iterated blow-up of a graph on $7$ vertices. More specifically, let $2a+2b+2c+d = 1$ and split vertices into $7$ parts---two of size $a$, two of size $b$, two of size $c$, and one of size $d$. Add edges between appropriate parts to make them complete bipartite and orient them as shown in the picture. Finally, iterate this process inside each of the seven parts. Note that the direction of edges between parts of size $c$ is not important, hence many non-isomorphic examples of graphs with this bound met can be found.}

\grafsekcja{8}{\graf{10}}{%
\uboundflagmatic{6}{81/512}}{%
81/512}{}{1}{Union of two balanced anticliques with random edges from the first to the second with probability~$\frac34$.}

\grafsekcja{9}{\graf{11}\ss\graf{12}}{%
Upper bound: \colorbox{gray!10}{$c_9$}, where 
$$c_9=\max_{x\in [0,\frac12]}\frac{32(1-2x)x^3}{1-(1-2x)^4}
=4-\frac{6}{\sqrt[3]{\sqrt{2}-1}}+6\sqrt[3]{\sqrt{2}-1}$$
and the maximum is attained at $x \approx 0.37346$. Proved by Huang \cite{Hu}, also included as an example application of the \flagmatic{} software \cite{Flagmatic}. 
}{%
$c_9 \approx 0.423570$}{}{1}{Split vertices into two parts, $A$ of size $x$ and $B$~of size $1-x$. Then, put all possible edges from $A$ to $B$ and iterate this process inside $A$.}

\grafsekcja{10}{\graf{13}\ss\graf{14}}{%
\uboundflagmatic{5}{81/400}}{%
81/400}{}{0.8}{Construction $\Gconstr{10}$ is the following: split vertices into $5$ parts---$A$ of size $\frac{3}{20}$, $B$~of size $\frac{1}{4}$, $C$ of size~$\frac{3}{20}$, $D$ of size $\frac{1}{4}$, and $E$ of size $\frac{1}{5}$. Put all edges from $A$ to~$B$, from $B$ to $C$, from $C$ to $D$, from $D$ to~$A$, from $C$ to~$E$, and from $A$ to $E$.} In \cite{HMNW} there~is also a~general probabilistic construction for any orientation of a star.

\grafsekcja{11}{\graf{141}}{%
\uboundflagmatic{5}{1/8}}{%
1/8}{}{1}{Balanced union of arbitrary two regular tournaments.}

\grafsekcja{12}{\graf{142}}{%
Upper bound: \colorbox{gray!10}{1/2}

As each weak component of this graph is a transitive tournament, the inducibility of this graph is equal to the inducibility of the non-oriented graph $K_3\sqcup \An{1}$, which is equal to the inducibility of $K_{1,3}$ determined in~\cite{BS}.}{%
1/2}{}{1}{Balanced union of two transitive tournaments.}

\grafsekcja{13}{\graf{15}}{%
Upper bound: \colorbox{gray!10}{2/21} proved in~\cite{HLPV} using flag algebra method with additional stability arguments.}{%
2/21}{}{1}{Iterated balanced blowup of $\overrightarrow{C_4}$.}

\grafsekcja{14}{\graf{16}}{%
\uboundflagmatic{5}{3/16}}{%
3/16}{}{1}{Balanced complete bipartite graph with randomly oriented edges. In~\cite{HLPV} it is additionally proven that every extremal graph is within edit distance $o(n^2)$ from a~pseudorandom orientation of a~balanced complete bipartite graph.}

\grafsekcja{15}{\graf{17}}{%
Upper bound: \colorbox{gray!10}{$c_{15}$} proved in~\cite{HLPV}, where 
$$c_{15}=\max_{x\in [0,\frac12]}\frac{12x^2(1-2x)^2}{1 - 2x^4}
= 9\left(\sqrt{2} - 2\right) + 6 \sqrt{2\left(\sqrt{2} - 1\right)}$$
and the maximum is attained at $x \approx 0.25202$. It is shown in \cite{HLPV} that the unique limit object maximizing inducibility of \graf{17} is given by the construction presented below. 
}{%
$c_{15} \approx 0.189000$}{}{1}{Construction $\Gconstr{15}$ is the following: split vertices into three parts, $A$ of size $x$, $B$~of size $1 - 2x$, and $C$ of size $x$. Put all possible edges from $A$ to $B$ and from $B$ to $C$. Iterate this process inside $A$ and inside $C$.}

\grafsekcja{16}{\graf{18}}{%
Upper bound: \colorbox{gray!10}{3/8}

It is equal to the inducibility of non-oriented $K_{2,2}$ solved in \cite{BNT,PG}.}{
3/8}{}{1}{Complete balanced bipartite graph with edges oriented from the first part to the second.}

\grafsekcja{17}{\graf{19}\ss\graf{20}}{%
\uboundflagmatic{6}{$\approx$ 0.095640}}{%
$2/21 \approx 0.095238$}{}{1}{Iterated balanced blow-up of graph \graf{19}.}

\grafsekcja{18}{\graf{21}\ss\graf{22}}{%
\uboundflagmatic{6}{$\approx$ 0.189030}}{%
$c_{18} \approx 0.189000$}{, where 
$$c_{18}=\max_{x\in [0,\frac12]}\frac{12x^2(1-2x)^2}{1 - 2x^4}
= 9\left(\sqrt{2} - 2\right) + 6 \sqrt{2\left(\sqrt{2} - 1\right)}$$
and is attained at $x \approx 0.25202$.
}{1}{

\noindent Construction $\Gconstr{18}$ is the following: split vertices into three parts, $A$ of size $x$, $B$~of size $x$, and $C$ of size $1 - 2x$. Put all possible edges from $A$ to $B$ and from $B$ to $C$, put also all possible edges inside $C$ to make it an arbitrary tournament. Iterate process inside $A$ and inside $B$.}

\grafsekcja{19}{\graf{23}\ss\graf{24}}{%
\uboundflagmatic{6}{$\approx$ 0.317681}}{%
$c_{19} \approx 0.317678$}{, where 
$$c_{19}=\max_{x\in [0,\frac12]}\frac{24(1-2x)x^3}{1-(1-2x)^4}
=\frac{3}{2}\left(2-\frac{3}{\sqrt[3]{\sqrt{2}-1}}+3\sqrt[3]{\sqrt{2}-1}\right)$$
and the maximum is attained at $x \approx 0.37346$.
}{1}{Construction $\Gconstr{19}$ is the following: split vertices into three parts---$A$ of size~$x$, $B$ of size $x$, and $C$ of size $1-2x$. Put all possible edges from $C$ to $A$, from $C$ to~$B$ and all possible edges in parts $A$ and $B$ to make them arbitrary tournaments. Iterate this process inside part $C$.}

\grafsekcja{20}{\graf{25}\ss\graf{26}}{%
\uboundflagmatic{6}{$\approx$ 0.119760}}{%
$c_{20} \gtrsim 0.119537$}{, where 
$$c_{20} = \max \frac{24y(1-x-y)}{1-x^4} (xy I_1 + y(1-x-y) I_2 + (1-x-y)^2 I_3),$$  
$$I_1 = \int_0^1 p(a)(1-p(a))\dd a, \quad I_2 = \int_0^1 \int_0^a p(a)^2p(b)(1-p(b))\dd b\dd a,$$
$$I_3 = \int_0^1 \int_0^a \int_0^b (1-p(c))p(b)(1-p(a))\dd c\dd b\dd a$$
and the maximum is taken over all $x,y \in[0,1]$ and functions $p : [0,1] \longrightarrow [0,1]$.
The above lower bound is obtained by maximizing over values of $x, y \in [0,1]$ and $p$ being a polynomial of degree at most $7$.}{1}{Construction $\Gconstr{20}$ is the following: split vertices into three parts---$A$ of size~$x$, $B$ of size $y$, and $C$ of size $1-x-y$, fix also a~function $p: [0,1] \to [0,1]$. Put all possible edges from $A$ to $B$, from $C$~to~$A$, and inside part $C$ to make it a~transitive tournament. Treat $C$ as a~finite subset of $[0,1]$, with order $<$ on $[0,1]$ preserving the transitive order of vertices of $C$. For each vertex $v \in C$ and $w \in B$, put an edge from $v$ to $w$ independently at random with probability $p(v)$. Finally, iterate this process inside part $A$. }

\grafsekcja{21}{\graf{27}}{%
\uboundflagmatic{5}{$(28+6\sqrt{3})/169$}}{%
$(28+6\sqrt{3})/169 \approx 0.227173$}{}{1.3}{Circular graph with parameter $(9+\sqrt{3})/26$.

To determine the number of embeddings of $\graf{27}$ to the circular graph $S^1(\alpha)$ with parameter $\alpha=(9+\sqrt{3})/26$, note that if the vertices of $\graf{27}$ are denoted by $v$, $x$, $y$, $z$ in such a way that $v\to x\to y\to z$, then every its embedding has four vertices in order $v$, $x$, $y$, $z$ along the circle. 
Fix $v$ on the circle and parameterize the remaining vertices by their oriented arc-distance from $v$.
Since $z$ is non-adjacent to $v$, its position must be in the interval $(\alpha, 1-\alpha)$, while $x$ and $y$ need to form an ordered pair in the interval of vertices adjacent to both $v$ and $z$.  
Therefore, the density of~$\graf{27}$ in $S^1(\alpha)$ is equal to
\[4! \int_{\alpha}^{1-\alpha}\frac{1}{2}(\alpha-(z-\alpha))^2\dd z = \frac{28+6\sqrt{3}}{169}.\]
}

\grafsekcja{22}{\graf{28}}{%
\uboundflagmatic{8}{$\approx$ 0.244055}

By Lemma \ref{lem:symmetrization_lemma}, we may assume that the graphs maximizing the number of induced copies of~$\graf{28}$ have the property that the in- and out-neighborhood of non-neighbors are the same. In particular, we can consider the inducibility of~$\graf{28}$ in a~family of graphs which have no induced copy of $T_2 \sqcup \An{1}$ (an edge plus an isolated vertex) or $\Pn{3}$, simplifying this way the computations in \flagmatic.}{%
$c_{22} \approx 0.244053$ }{, where
$$c_{22}=\max_{0\leq 2y\leq 1-q\leq 1} \left(1-\frac{2y}{1-q}\right)^2\frac{12y^2}{(1-q)^2}+\frac{24qy^3}{(1-q)^2}\left(\frac{1}{1+q+q^2}-\frac{y}{1-q^4}\right).$$
}{.7}{Construction consists of a~sequence of blobs $A^i$ for $i \in \mathbb{Z}$, where each $A^i$ is an anticlique, and each pair of vertices from different blobs is joined by an arc pointing to the blob of the larger index. 
Let blob $A^i$ be of the size $a_{|i|}$, where $a_0 + 2a_1 + 2a_2 + \ldots =1$. Letting the blob sizes decrease geometrically, i.e., $a_i=yq^{i-1}$ for $i=1,2,\ldots$ and $a_0 = 1 - \frac{2y}{1-q}$, we obtain the desired density.
}

\grafsekcja{23}{\graf{29}\ss\graf{30}}{%
\uboundflagmatic{6}{$\approx$ 0.177784}}{%
$c_{23} \approx 0.177630$}{, where 
\begin{eqnarray*}
c_{23}&=&\max_{x\in[0,1]}\frac{\frac{8}{5}x(1-x)^3 + \frac{8}{315}(1-x)^4}{1-x^4} \\
&=& \frac{4}{105}\left(61^{2/3}\sqrt[3]{\sqrt{7690} - 63} - 61^{2/3}\sqrt[3]{63 + \sqrt{7690}} + 42\right)
\end{eqnarray*}
and the maximum is attained at $x \approx 0.24063$.
}{1}{Construction $\Gconstr{23}$ is the following: split vertices into two parts---$A$ of size $x$ and $B$ of size $1-x$. Put edges in $B$ to make it an iterated blow-up of $\overrightarrow{C_4}$ and all possible edges from $A$ to $B$. Iterate the process inside $A$.}

\grafsekcja{24}{\graf{31}\ss\graf{32}}{
Upper bound: \colorbox{gray!10}{$3/8$}

\begin{proof} Let $G$ be a~graph on $n$ vertices maximizing the number of induced copies of $H = \graf{31}$ and introduce the following equivalence relation in $V(G)$: $v \sim w$ iff $N^+(v) = N^+(w)$ and $N^-(v) = N^-(w)$. By Lemma \ref{lem:symmetrization_lemma}, since $H$ has the property that $v \sim w$ iff $v$ and $w$ are not joined by an arc, we may assume that $G$ satisfies the same property.

Denote by $T$ the set of all vertices which belong to equivalence classes of size one, and let $t = \abs{T}$. Let $m$ be the number of the remaining equivalence classes; denote them by $B_1, \ldots, B_m$, where $b_i = \abs{B_i}$ and $b_1 \leq b_2 \leq \ldots \leq b_m$. Note that each induced copy of $H$ which is not disjoint from $T$ contains exactly two vertices in $T$ (joined by an arc, which can be oriented in any way) and two vertices in some $B_i$, with arcs oriented from $T$ to $B_i$. Hence, the number of induced copies of~$H$ in $G$ does not depend on the orientation of arcs between vertices in $T$. Also, by orienting the arcs in such a~way that $T \to B_i$ for every $i \in [m]$, we won't decrease the number of induced copies of~$H$.
The following claim gives the orientation of arcs between the remaining pairs of equivalence classes.

\begin{clm} We can orient the arcs between $B_i$ and $B_j$ for every $i < j$ so that $B_i \to B_j$ without decreasing the number of induced copies of~$H$. 
\end{clm}
\begin{proof}
Consider a partition $I \cup J = [m-1]$ of indices such that $B_i \to B_m$ for $i \in I$ and $B_m \to B_j$ for $j \in J$. Let us modify graph $G$ by orienting all arcs incident to $B_j$ and $B_m$ towards $B_m$ for each $j \in J$. Each copy of~$H$ removed in this way has exactly one vertex in $B_m$, exactly two vertices in $B_j$ for some $j \in J$, and one more vertex which belongs neither to $B_j$ nor $B_m$ and has outgoing arcs to $B_j$. The number of lost copies is then equal to

\begin{align}\label{eq:g24-first-expr}
b_m  \sum_{j \in J} \binom{b_j}{2} \left(t + \sum_{i \in I, B_i \to B_j} b_i + \sum_{j'\in J, B_{j'} \to B_j} b_{j'}\right) .
\end{align}

The created copies of $H$ contain exactly two vertices in $B_m$ and at least one vertex in some $B_j$ for $j \in J$, so the number of created copies is equal to

\begin{align}\label{eq:g24-second-expr}
\binom{b_m}{2} \left( t\sum_{j \in J} b_j + \sum_{i \in I, j \in J} b_i b_j  + \sum_{j < j' \in J} b_j b_{j'}  \right).
\end{align}

Now, we can compare both values by looking at appropriate terms in both formulas. For each pair $j \neq j' \in J$, we have exactly one of the following two terms
$$ b_m  \binom{b_j}{2}b_{j'}, \quad b_m\binom{b_{j'}}{2}b_j $$
in \eqref{eq:g24-first-expr} containing $b_jb_{j'}$, and exactly one term
$$\binom{b_m}{2} b_j b_{j'} $$ 
in \eqref{eq:g24-second-expr}. Since $b_m \geq b_i$ for all $i \in [m-1]$, we have 
  $$\binom{b_m}{2} b_j b_{j'} \geq b_m  \binom{b_j}{2}b_{j'} \quad \mathrm{and} \quad \binom{b_m}{2} b_j b_{j'} \geq b_m\binom{b_{j'}}{2}b_j.$$
  
Now, for each $i \in I$ and $j \in J$, we have at most one term containing $b_ib_j$ in~\eqref{eq:g24-first-expr}, it is equal to $b_m \binom{b_j}{2} b_i$, which is not greater than $\binom{b_m}{2} b_ib_j$ in \eqref{eq:g24-second-expr}. Finally, for $t$ and any index $j \in J$, we have one term $b_m \binom{b_j}{2} t$ in \eqref{eq:g24-first-expr}, which is again not greater than
$\binom{b_m}{2} b_j t$ in \eqref{eq:g24-second-expr}. Therefore, \eqref{eq:g24-first-expr} is not greater than \eqref{eq:g24-second-expr}, i.e., the number of copies of~$H$ in $G$ did not decrease after orienting all arcs towards~$B_m$.

We can repeat this argument for the subgraph of $G$ induced by $T$ and $B_1, \ldots, B_{m-1}$---just note that by reorienting edges in this subgraph, we will not decrease the number of copies of~$H$ which contain at least one vertex in $B_m$. By repeating this iterative procedure, we will eventually obtain the property that $B_i \to B_j$ whenever $i < j$.
\end{proof}

So far, we proved that the graph $G$ has the following structure:

$$\overrightarrow{T \to B_1 \to B_2 \to \ldots \to B_m}.$$

We would like to show that $m = 1$, so that $G$ is just of the form $T \to B_1$. Assume that this is not the case, i.e., $m \geq 2$, and $G$ is an extremal graph with $b_m$ being maximal and---for this value of $b_m$---also $t$ being maximal.

\begin{clm}\label{cla:tbm}
If $m \geq 2$, then $t > b_m$.
\end{clm}
\begin{proof}
Since $G$ is an extremal graph, introducing a~single arc in $B_1$ shall not increase the number of copies of~$H$. Such an arc removes $\binom{t}{2}$ copies of~$H$ and creates $\sum_{i=2}^m \binom{b_i}{2}$ copies of~$H$. Therefore, we have

\begin{align*}
\binom{t}{2} \geq \sum_{i \geq 2} \binom{b_i}{2}.
\end{align*}

In particular, $t \geq b_m$. Moreover, if $m > 2$, then $t > b_m$, since $\binom{b_2}{2} \geq 1$. Also, if $m = 2$ and $t = b_m$, then we can move those two vertices which we just joined by an arc to $T$; this way, we will not decrease the number of copies of~$H$ in $G$, but we will increase the size of $T$, contradicting our assumptions about the maximality of the graph $G$. It follows that we must have $t > b_m$.
\end{proof}

We are ready to give an argument contradicting the maximality of $G$. For this, remove one vertex from $T$ and add one vertex to $B_m$. This way, we destroy the following number of copies of~$H$:
\begin{align*}
(t-1)\sum_{i = 1}^m \binom{b_i}{2} + \sum_{i=1}^{m-1} \sum_{j=i+1}^m b_i \binom{b_j}{2},
\end{align*}
which can be written in the following way:
\begin{align}\label{eq:gr24-third-sum}
(t-1)\binom{b_m}{2} + (t-1)\sum_{i=1}^{m-1} \binom{b_i}{2} + \sum_{i=1}^{m-1} \sum_{j=i+1}^{m-1} b_i \binom{b_j}{2} + \sum_{i=1}^{m-1} b_i \binom{b_m}{2}.
\end{align}
On the other hand, we create the following number of copies of $H$:
\begin{align*}
b_m\left(\binom{t-1}{2} + (t-1)\sum_{i=1}^{m-1}b_i + \sum_{i=1}^{m-1}\sum_{j=i+1}^{m-1} b_ib_j \right),
\end{align*}
which can be written in the following way:
\begin{align}\label{eq:gr24-fourth-sum}
b_m\binom{t-1}{2} + (t-1)\sum_{i=1}^{m-1} \frac{b_i b_m}{2} + \sum_{i=1}^{m-1}\sum_{j=i+1}^{m-1} b_ib_jb_m + \sum_{i=1}^{m-1} (t-1)\frac{b_i b_m}{2}.
\end{align}
Since $t-1 \geq b_m$ from Claim~\ref{cla:tbm}, it is straightforward to see that each summand of \eqref{eq:gr24-third-sum} is not greater than an appropriate summand of \eqref{eq:gr24-fourth-sum}, and that there is at least one strictly smaller summand (e.g. the last one).

It follows that graph $G$ must be of the form $T \to B_1$. The number of copies of $H$ in $G$ is then equal to $\binom{t}{2}\binom{n - t}{2}$, which is maximized for $t \in \{ \lfloor \frac n2 \rfloor, \lceil \frac n2 \rceil\}$. This gives the desired upper bound for the inducibility of $H$.
\end{proof}
}{%
3/8}{}{1}{Full join from an arbitrary tournament to an anticlique of the same size.}

\grafsekcja{25}{\graf{33}}{%
Upper bound: \colorbox{gray!10}{$4/9$}

\begin{proof}
Note that in each copy of $\graf{33}$ there are exactly two vertices for which outdegree, indegree, and nondegree are equal to $1$. Basing on this observation and using AM-GM inequality, we get the following bound on the number of copies of $\graf{33}$ in any $n$-vertex graph $G$:
\[\Nnum{\graf{33}}{G} \leq \frac12\sum_{v\in V(G)}d^+(v)d^-(v)d'(v)\leq\frac{n(n-1)^3}{54}.\]
In particular, 
\[\Ilim{\graf{33}} \leq \lim_{n \to \infty} \frac{n(n-1)^3}{54}/\binom{n}{4} = \frac49.\]
\end{proof}}{%
4/9}{}{0.9}{Balanced blow-up of $\overrightarrow{C_3}$.}

\grafsekcja{26}{\graf{34}\ss\graf{35}}{%
\uboundflagmatic{6}{$\approx$ 0.113205}}{%
$c_{26} \gtrsim 0.112567$}{, where
$$c_{26} = \max \frac{24y^2(1-x-y)}{1-x^4} (x I_1 + (1-x-y) I_2),$$
$$I_1 = \int_0^1 p(a)(1-p(a))\dd a, \quad I_2 = \int_0^1 \int_0^a (1-p(b))^2p(a)(1-p(a))\dd b\dd a$$
and the maximum is taken over all $x,y \in [0,1]$ and functions $p : [0,1] \longrightarrow [0,1]$.
The above lower bound is obtained by maximizing over $x, y \in [0,1]$ and $p$ being a polynomial of degree at most $7$.}{1}{Construction $\Gconstr{26}$ is the following: split vertices into three parts---$A$ of size~$x$, $B$ of size $y$, and $C$ of size $1-x-y$, fix also a~function $p: [0,1] \to [0,1]$. Put all possible edges from $B$ to $A$, from $A$ to $C$, and inside part $C$ to make it a~transitive tournament. Treat $C$ as a~finite subset of $[0,1]$, with order $<$ on $[0,1]$ preserving the transitive order of vertices of $C$. For each vertex $v \in C$ and $w \in B$, put an edge between $v$ and $w$ oriented independently at random, with probability $p(v)$ from $v$ to $w$ and with probability $1-p(v)$ from $w$ to $v$. Finally, iterate this process inside part $A$. }

\grafsekcja{27}{\graf{36}}{%
\uboundflagmatic{6}{$\approx$ 0.148148}}{%
$4/27 \approx 0.148148$}{}{1.3}{Circular graph with parameter $4/9$.

The exact value of the upper bound found by \flagmatic{} is equal to $\frac{491913175465823271551}{3320413933267719290880}$, which decimal expansion agrees with $4/27$ up to the tenth decimal place. 
We believe it is possible to obtain the proof of the optimal upper bound by providing eigenvectors for the eigenvalue zero and solving the semidefinite problem with good enough precision.
Unfortunately, computations on 6~vertices generate too large numerical errors on \texttt{csdp}, while computations on a~solver with higher precision require too large computational power. Thus, we were unable to prove the optimal bound. 

To determine the number of embeddings of $\graf{36}$ to the circular graph $S^1(4/9)$, note that each embedding of a directed triangle can be extended to an embedding of $\graf{36}$ by adding any vertex non-adjacent to one of the vertices of the triangle, and the set of such vertices has measure $1/3$. Moreover, each embedding of $\graf{36}$ is counted this way exactly once. 

In order to find the number of embeddings of a directed triangle with edges $v\to x\to y\to v$, we fix $v$ on the circle and parameterize the remaining vertices by their oriented arc-distance from~$v$. To create a directed triangle, the position of $x$ must be in the interval $(1/9,4/9)$, while $y$ needs to be in the interval of vertices appropriately connected to $v$ and $x$. Since this way we count each directed triangle three times, we need to divide the obtained value by $3$.

Therefore, the density of $\graf{36}$ in $S^1(4/9)$ is equal to
\[\frac{4!}{9}\int_{1/9}^{4/9} \left(x+\frac{4}{9} - \frac{5}{9}\right)\dd x = \frac{4}{27}.\]
}

\grafsekcja{28}{\graf{38}\ss\graf{39}}{%
Upper bound: \colorbox{gray!10}{$c_{28}$} proved in \cite{BLPP} (Theorem 4).
}{%
$c_{28} \approx 0.157501$}{, where
$$c_{28}=\max_{x\in[0,1]}\frac{x^4/8 + x^3(1-x)}{1-(1-x)^4} = \frac{8 - 3^{7/3} + 3^{5/3}}{8}$$
and the maximum is attained at $x \approx 0.85642$. It is proved in \cite{BLPP} that the extremal construction for sufficiently large number of vertices is precisely the one presented below.
}{1}{Construction $\Gconstr{28}$ is the following: split vertices into two parts---$A$ of size $x$ and $B$ of size $1-x$. Let $A$ be a~random tournament and put all possible edges from $A$ to $B$. Finally, iterate the process inside $B$.}

\grafsekcja{29}{\graf{37}}{%
Upper bound: \colorbox{gray!10}{$1/2$}
\begin{proof}
Let $G$ be an $n$-vertex graph with the maximum number of copies of~$\graf{37}$. As adding an arc between two non-neighbors may only increase the number of copies of $\graf{37}$, we may assume that $G$ is a tournament, i.e., $d^+(v)+d^-(v)=n-1$ for every $v\in V(G)$. Note that
\[\Nnum{\graf{37}}{G}\leq\frac12\sum_{v\in V(G)}\Nnum{\graf{143}}{G-v},\]
as in each copy of $\graf{37}$ there are two ways to select vertex $v$ in such a way that remaining three vertices form a copy of $\graf{143}$. Furthermore, for every tournament~$H$ on $k$ vertices, it holds
\[2\Nnum{\graf{143}}{H}+\binom{k}{3}=3\Nnum{\graf{143}}{H}+\Nnum{\graf{144}}{H}=\sum_{v\in V(H)}d^+(v)d^-(v)\leq \frac{k(k-1)^2}{4}\]
by the AM-GM inequality, and in consequence
\[\Nnum{\graf{143}}{H}\leq\frac{(k-1)k(k+1)}{24}.\] Applying this inequality for $k=n-1$ and $H=G-v$ for every $v\in V(G)$, and plugging to the previous estimation, we finally get
\[\Nnum{\graf{37}}{G}\leq\frac{n^2(n-1)(n-2)}{48},\]
hence
\[\Ilim{\graf{37}} \leq \lim_{n \to \infty} \frac{n^2(n-1)(n-2)}{48}/\binom{n}{4} = \frac12.\]
\end{proof}
An independent proof can be found in \cite{BLPP}.}{%
1/2}{}{1}{Circular graph with parameter $1/2$.}

\grafsekcja{30}{\graf{40}}{%
Upper bound: \colorbox{gray!10}{1}}{%
1}{}{1}{Transitive tournament.}

\section*{Acknowledgment}

We would like to thank Jakub Sliacan for his help with the \flagmatic{} software and the referees for their valuable comments which greatly improved the presentation of our results.

\appendix
\newpage
\section{Explanation of the codes}\label{sec:code}

Together with the preprint of the paper at  \hyperlink{https://arxiv.org/abs/2010.11664}{arXiv:2010.11664} we provided source codes of the programs (files with the extension \texttt{.sage}) which we used to obtain the upper bounds. All those programs require to have installed \texttt{sage} \cite{sage} and \flagmatic{} software \cite{Flagmatic}. Below we include a~short explanation of the the codes. As an example we consider the \texttt{graph14.sage} file:

\begin{lstlisting}[language=Python]
from flagmatic.all import *
P=OrientedGraphProblem(5, density="4:12233414", type_orders=[3])
P.set_extremal_construction(field=QQ, target_bound=3/16)
P.add_sharp_graphs(0,9,10,11,12,13,156,157,158,159,167,168,169,170,171,180,181,182,187)
#3:
P.add_zero_eigenvectors(0,matrix(QQ, [(8,0,0,0,0,0,0,0,0,0,0,0,0,0,0,0,0,0,0,1,1,1,1,1,1,1,1), (0,0,0,0,0,0,0,0,0,0,0,0,0,0,0,0,0,0,0,0,0,-1,-1,1,1,0,0), (0,0,0,0,0,0,0,0,0,0,0,0,0,0,0,0,0,0,0,0,-1,0,-1,1,0,1,0), (0,0,0,0,0,0,0,0,0,0,0,0,0,0,0,0,0,0,0,-1,0,0,1,-1,0,0,1)]))
#3:1232
P.add_zero_eigenvectors(2,matrix(QQ, [(0,0,0,0,0,0,0,0,0,-1,0,0,0,0,1,0,0,0,0,0,0,0,0,0,0,0,0), (0,0,0,0,0,0,0,0,0,0,-1,0,0,1,0,0,0,0,0,0,0,0,0,0,0,0,0), (0,0,0,2,2,0,0,0,0,1,1,0,0,1,1,0,0,0,0,0,0,0,0,0,0,0,0), (0,0,0,1,-1,0,0,0,0,0,0,0,0,0,0,0,0,0,0,0,0,0,0,0,0,0,0)]))
#3:1223
P.add_zero_eigenvectors(3,matrix(QQ, [(0,0,0,2,2,0,0,0,0,1,1,0,0,1,1,0,0,0,0,0,0,0,0,0,0,0,0), (0,0,0,1,-1,0,0,0,0,0,0,0,0,0,0,0,0,0,0,0,0,0,0,0,0,0,0), (0,0,0,0,0,0,0,0,0,-1,0,0,0,0,1,0,0,0,0,0,0,0,0,0,0,0,0), (0,0,0,0,0,0,0,0,0,0,-1,0,0,1,0,0,0,0,0,0,0,0,0,0,0,0,0)]))
#3:1213
P.add_zero_eigenvectors(4,matrix(QQ, [(0,2,2,0,0,0,0,0,0,0,0,0,0,0,0,1,1,1,1,0,0,0,0,0,0,0,0), (0,1,-1,0,0,0,0,0,0,0,0,0,0,0,0,0,0,0,0,0,0,0,0,0,0,0,0), (0,0,0,0,0,0,0,0,0,0,0,0,0,0,0,1,0,0,-1,0,0,0,0,0,0,0,0), (0,0,0,0,0,0,0,0,0,0,0,0,0,0,0,0,1,-1,0,0,0,0,0,0,0,0,0)]))
P.solve_sdp()
P.make_exact(16)
P.write_certificate("graph14.cert")
\end{lstlisting}

In line~1, we import \flagmatic{} libraries. In line 2, using the function \texttt{OrientedGraphProblem} we define the problem $P$ we want to solve. First two parameters of the function are mandatory. The first one defines the maximum size of flags used in the computations, while the second one defines the graph density we want to maximize. In the considered example, \texttt{``4:12233414''} means an oriented graph on $4$ vertices (labeled with numbers from $1$ to $4$) and edges $12$, $23$, $34$, $14$, i.e., graph \graf{16}. We can also specify other conditions for the problem, for instance in the above code we limit the computations only to types of order 3. 

At this moment we have already defined the problem which \flagmatic{} can solve numerically using a semidefinite programming solver. However, in order to obtain an exact result we need to provide more information for the rounding procedure. In line 3, we define the target bound (i.e., the value which we believe is the correct upper bound) and over which field we perform the computations --- here, \texttt{QQ} stands for the field of rational numbers. In line 4, we specify which graphs densities are forced to satisfy the bound as equalities, they follow from the expected extremal constructions. Finally, in lines 5-12, we provide additional eigenvectors corresponding to the zero eigenvalue needed to properly round the semidefinite matrix. Some of the vectors follow from the extremal constructions, but some of them were obtained using the method described in the appendix of \cite{BHLPUV} and in Section 2.4.2.2 of \cite{Baber}.

All that remains to do is to perform the computations. In line 13, we numerically solve the semidefinite programming (by default it is performed by \texttt{csdp} \cite{csdp}, but in a~few cases where double-double precision is needed we used \texttt{sdpa-dd} \cite{sdpadd}). In line 14, we round the output of the solver in the considered field. Finally, in the last line, we create a~certificate which contains a~complete description of the problem, used flags, graphs, matrices, and the proven bound. In Appendix~\ref{sec:inspect_certificate} we explain how one can verify the obtained proof using such a certificate. 

To run the program, use the command \texttt{sage graph14.sage}.

\section{Verification of the proofs}\label{sec:inspect_certificate}

For all the results obtained by the \flagmatic{} software we provide certificates containing a~complete description of the problem, used flags, graphs, matrices, and the proven bound. Since a proper verification usually requires too many computations to do it by hand, \flagmatic{} is distributed together with an independent checker program \texttt{inspect\_certificate.py} written by Vaughan. The program by principle requires only Python to run, however, some of its functionalities, including computations in different number fields, are only available when used with the \texttt{sage} software \cite{sage}.
We include a~copy of this program together with all the certificates at \hyperlink{https://arxiv.org/abs/2010.11664}{arXiv:2010.11664}. This copy has additional line in the code, which tells \texttt{sage} to simplify computations made in non-rational fields.

In order to verify the obtained upper bound using \texttt{inspect\_certificate.py}, assuming the certificate file (for instance \texttt{graph14.cert}) is in the same folder as \texttt{inspect\_certificate.py}, one needs to run the following command:

\texttt{sage inspect\_certifate.py graph14.cert -{}-verify-bound}

The checker program \texttt{inspect\_certificate.py} has many other functionalities. In particular, it can display problem description, used graphs, flags, or matrices in a human-readable form. A~detailed explanation of the contents of certificates and how the checker program works can be found in \cite{FV}.

\end{document}